\documentclass[a4paper, 10pt, twoside, notitlepage]{amsart}

\usepackage{amsmath,amscd}
\usepackage{amssymb}
\usepackage{amsthm}
\usepackage{comment}
\usepackage{graphicx, xcolor}

\usepackage{mathrsfs}
\usepackage[ocgcolorlinks, linkcolor=blue]{hyperref}

\usepackage{bm}
\usepackage{bbm}
\usepackage{url}

\usepackage[utf8]{inputenc}
\usepackage{mathtools,amssymb}
\usepackage{esint}
\usepackage{tikz}
\usepackage{dsfont}
\usepackage{relsize}
\usepackage{url}
\usepackage{xcolor}
\usepackage{graphicx}
\usepackage{mathrsfs}
\usepackage[shortlabels]{enumitem}
\usepackage{lineno}
\usepackage{amsmath}
\usepackage{enumitem}
\usepackage{amsthm} 
\usepackage{verbatim}
\usepackage{dsfont}
\numberwithin{equation}{section}

\renewcommand{\H}{{\mathbb H}}

\allowdisplaybreaks

\graphicspath{{images/}}

\newtheorem{theorem}{Theorem}[section]
\newtheorem{lemma}[theorem]{Lemma}
\newtheorem{definition}[theorem]{Definition}
\newtheorem{corollary}[theorem]{Corollary}

\newtheorem{proposition}[theorem]{Proposition}

\newtheorem{remark}[theorem]{Remark}

\title[Calder\'on's problem for logarithmic Laplacian]{The Calder\'on problem for the logarithmic Schr\"odinger equation}

\author[B. Harrach]{Bastian Harrach}
\address{Institute for Mathematics, Goethe University Frankfurt, Germany}
\email{harrach@math.uni-frankfurt.de}

\author[Y.-H. Lin]{Yi-Hsuan Lin}
\address{Department of Applied Mathematics, National Yang Ming Chiao Tung University, Hsinchu, Taiwan \& Fakult\"at f\"ur Mathematik, University of Duisburg-Essen, Essen, Germany}
\curraddr{}
\email{yihsuanlin3@gmail.com}

\author[T. Weth]{Tobias Weth}
\address{Institute for Mathematics, Goethe University Frankfurt, Germany}
\email{weth@math.uni-frankfurt.de}



\newcommand{\loglap}{L_{\text{\tiny $\!\Delta \,$}}\!}

\newcommand{\R}{{\mathbb R}}

\newcommand{\N}{{\mathbb N}}

\newcommand{\eps}{\epsilon}

\newcommand {\p} {\partial}

\newcommand{\LC}{\left(}
\newcommand{\RC}{\right)}
\newcommand{\wt}{\widetilde}

\newcommand{\norm}[1]{\lVert #1 \rVert}
\newcommand{\abs}[1]{\left\lvert #1 \right\rvert}
\DeclareMathOperator{\dist}{dist} 

\begin{document}

	\maketitle
	\begin{abstract}
		
		We study the Calder\'on problem for a logarithmic Schr\"odinger type operator of the form $\loglap +q$, where $\loglap$ denotes the logarithmic Laplacian, which arises as formal derivative $\frac{d}{ds} \big|_{s=0}(-\Delta)^s$ of the family of fractional Laplacian operators. This operator enjoys remarkable nonlocal properties, such as the unique continuation and Runge approximation. Based on these tools, we can uniquely determine bounded potentials using the Dirichlet-to-Neumann map. Additionally, we can build a constructive uniqueness result by utilizing the monotonicity method. Our results hold for any space dimension. 
		
		\medskip

	\end{abstract}

	\tableofcontents

	\section{Introduction}\label{sec: introduction}
	
	Calder\'on's pioneer work \cite{Calderon_80} investigated the inverse conductivity problem, which asks whether the conductivity can be determined from its boundary electrical voltage and current measurements. Using a suitable reduction scheme shows that a closely related problem is the Calder\'on problem for the classical Schr\"odinger equation, which can be stated as follows. Let $\Omega \subset \R^n$ be a bounded domain with Lipschitz boundary $\p \Omega$, and $q\in L^\infty(\Omega)$. Consider the Dirichlet boundary value problem
	\begin{equation}\label{eq: local Schro}
		\begin{cases}
			\LC-\Delta +q \RC u=0 &\text{ in }\Omega, \\
			u=f &\text{ on }\p \Omega.
		\end{cases}
	\end{equation}
	Assume that $0$ is not a Dirichlet eigenvalue of the Schrödinger operator $-\Delta+q$ in $\Omega$, which implies that there exists a unique solution $u\in H^1(\Omega)$ to \eqref{eq: local Schro}, for any given Dirichlet boundary data $f\in H^{1/2}(\p \Omega)$. Then the boundary measurement is called the (full) \emph{Dirichlet-to-Neumann} (DN) map, which is well-defined and can be encoded by 
	\begin{equation}\label{local DN}
		\Lambda_q^{(1)} : H^{1/2}(\p \Omega)\to H^{-1/2}(\p \Omega), \quad f\mapsto  \left.\p_\nu u_f \right|_{\p \Omega},
	\end{equation}
	where $u_f \in H^{1}(\Omega)$ is the solution to \eqref{eq: local Schro}. The Calder\'on problem for the Schr\"odinger equation \eqref{eq: local Schro} is to ask whether the DN map $\Lambda_q$ determines the potential $q\in L^\infty(\Omega)$ or not. The result is positive for a more regular potential $q$ with the full boundary DN map, and we refer readers to the articles \cite{SU87} for $n\geq 3$ and \cite{Buk08} for $n=2$. The article \cite{uhlmann2009electrical} provides a comprehensive introduction and review in this direction.

	After several decades of developments on Calder\'on's problems, this type of problem has been generalized to the nonlocal scenario. The Calder\'on problem for the \emph{fractional Schr\"odinger equation} has first been studied in \cite{GSU20}. The underlying problem was proposed as an exterior value problem due to its natural nonlocality: Given $s\in (0,1)$, let $\Omega \subset\R^n$ be a bounded open set with Lipschitz boundary for $n\in \N$, and let $q\in L^\infty(\Omega)$. Consider the Dirichlet exterior value problem 
	\begin{equation}\label{eq: fractional Schro}
		\begin{cases}
			\LC (-\Delta)^s +q \RC u=0 &\text{ in }\Omega, \\
			u=f &\text{ in }\Omega_e,
		\end{cases}
	\end{equation}
	where 
	\[
	\Omega_e \vcentcolon =\R^n\setminus \overline{\Omega}.
	\]
	Similarly as in the case $s=1$, let us assume that $0$ is not a Dirichlet eigenvalue of $(-\Delta)^s+q$ in $\Omega$. Then there exists a unique solution $u\in H^s(\R^n)$ to \eqref{eq: fractional Schro}, for any given Dirichlet exterior data $f$ in a suitable function space, such as $C^\infty_c(\Omega_e)$. The exterior (partial) DN map can be formally defined by 
	\begin{equation}\label{nonlocal DN}
		\Lambda_q^{(s)} : C^\infty_c(W_1)\ni f\mapsto  \left. (-\Delta)^s u_f \right|_{W_2},
	\end{equation}
	where $W_1, W_2\Subset \Omega_e$ can be arbitrary nonempty open subsets. 
	
	In the foundational work \cite{GSU20}, the authors showed that the DN map \eqref{nonlocal DN} can determine the bounded potential $q$ in $\Omega$ uniquely, which holds for any spatial dimension $n\in \N$. 
	Moreover, inverse problems for fractional equations have more profound properties than their local counterparts, such as the \emph{unique continuation property} (UCP) and \emph{Runge approximation property}. Roughly speaking, the UCP for $(-\Delta)^s$ ($0<s<1$) states that given a nonempty open subset $\mathcal{O}\subset \R^n$,  
	$$
	u=(-\Delta)^su =0 \text{ in }\mathcal{O}\quad \text{ implies that }\quad u\equiv 0\text{ in } \R^n,
	$$ 
	where $u$ belongs to appropriate function spaces (e.g., $u$ belongs to some fractional Sobolev space $H^r(\R^n)$ for some $r\in \R$). Moreover, the Runge approximation property states that any $L^2$ function can be approximated by solutions of the fractional Schr\"odinger equation.
	
	Based on the properties as mentioned above, fractional-type inverse problems have received rapidly growing attention in recent years. In the work \cite{cekic2020calderon}, the authors demonstrated that both drift and potential can be determined uniquely, which cannot be true for the local case in general. Moreover, simultaneous recovering for both the obstacle and surrounded medium has been studied in \cite{CLL2017simultaneously}, and the determination of bounded potentials for anisotropic nonlocal Schr\"odinger equation was investigated in \cite{GLX}.  These two problems remain open for the case $s=1$ and $n\geq 3$. Therefore, one could expect that the nonlocality is beneficial in the study of related inverse problems.

	For more related work on both linear and nonlinear nonlocal inverse problems, we refer readers to \cite{harrach2017nonlocal-monotonicity,harrach2020monotonicity,LL2022inverse,LL2020inverse,GRSU20,CMRU20,RS20,ruland2018exponential,GRSU20,LLR2019calder,lin2020monotonicity,LZ2023unique,GU2021calder,LLU2022para} and references therein. In particular, one can also determine the interior coefficients, by using either a reduction of the Caffarelli-Silvestre extension (see e.g. \cite{CGRU2023reduction,ruland2023revisiting,LLU2023calder,LZ2024approximation}) or the heat semigroup on closed Riemannian manifolds (see e.g. \cite{feizmohammadi2021fractional,Fei24_TAMS,FKU2024calder,lin2024fractional}), where the UCP may not be a necessary tool for some cases. Very recently, an entanglement principle for nonlocal operators has also been investigated in \cite{FKU2024calder,FL24}, and it may significantly influence new nonlocal inverse problems.

	Loosely speaking, most of the existing studies of inverse problems are related to the classical Laplacian $-\Delta$ or the fractional Laplacian $(-\Delta)^s$, $s \in (0,1)$. We want to point out that the key tools in solving classical and fractional inverse problems:
	\begin{itemize}
		\item \textbf{Classical case.} Suitable integral identities and complex geometrical optics (CGO) solutions ($(-\Delta)^s$ for $s =1$).
		
		\item \textbf{Fractional case.} Suitable integral identities and Runge approximation property ($(-\Delta)^s$ for $0<s<1$).
	\end{itemize}
	The above-mentioned tools are usually used to recover lower-order coefficients in related mathematical problems. Furthermore, a recent work \cite{CdHS24} investigates geometric optics solutions for the fractional Helmholtz equation, which combines both local and nonlocal features.\\

	While the nonlocality distinguishes the fractional case from the classical one, it is important to note that the fractional Laplacian $(-\Delta)^s$ is still an elliptic operator of positive order $2s$ and therefore admits a highly useful elliptic regularity theory. In the present paper, we wish to study an inverse problem beyond the framework of operators of positive order. To motivate our problem, recall that   
	\begin{equation}\label{eq: limit of fra-Lap}
		\lim_{s\to 1^-}(-\Delta)^s u = -\Delta u \quad \text{and}\quad \lim_{s\to 0^+} (-\Delta)^s u =u \qquad \text{pointwisely in $\R^n$} 
	\end{equation}
	for any $u\in C^2_c(\R^n)$. For the second limit, one may identify a well-defined first-order correction term given as the formal derivative 	
	$\log (-\Delta)u  :=   \frac{d}{ds} \big|_{s=0}(-\Delta)^su$. The associated {\em logarithmic Laplacian operator}
	$$
	\loglap := \log (-\Delta)
	$$
	has been introduced in \cite{CW2019dirichlet} in the context of Dirichlet problems, and it has attracted growing interest in recent years due to its usefulness in the analysis of order-dependent families of linear and nonlinear fractional Dirichlet problems and their asymptotic limits as $s \to 0^+$, see e.g. \cite{CW2019dirichlet,JSnW20,feulefack-jarohs-weth,hernandez-santamaria-saldana,laptev-weth}. It also appears in the context of the $0$-fractional perimeter, see \cite{de-luca-novaga-ponsiglione}.

	We note that the operator $\loglap$ has the Fourier symbol $2\log \abs{\xi}$, which can be seen by differentiating the symbol $\abs{\xi}^{2s}$ with respect to $s$ and evaluate at $s=0$. More precisely, by \cite[Theorem 1.1]{CW2019dirichlet}, there holds 
	\begin{equation}\label{eq: def log-Lap Fourier}
		\log (-\Delta) u \vcentcolon = \mathcal{F}^{-1} \LC 2\log \abs{\xi}\widehat{u}(\xi)\RC , \quad \text{for} \quad u \in C^\alpha_c (\R^n),
	\end{equation}
	for some $\alpha>0$, where $\widehat{u}$ is the Fourier transform of $u$, and $\mathcal{F}^{-1}$ is the inverse Fourier transform. Due to the weak growth of the logarithmic symbol, $\loglap$ is sometimes called a (near) zero-order operator.
	
	Very recently, the problem of finding an extension problem for the logarithmic Laplacian has been solved in \cite{CHW2023extension}, and the authors used it to prove the UCP for $\loglap$. The UCP will play a major role in the proof of our main results.\\

	\noindent \textbf{Mathematical formulations and main results.}
	Let $\Omega\subset \R^n$ be a bounded domain with Lipschitz boundary $\p \Omega$, and $q\in L^\infty(\Omega)$. Consider 
	\begin{equation}\label{eq: main equation}
		\begin{cases}
			\LC \loglap  + q \RC  u=0 &\text{ in }\Omega,\\
			u=f &\text{ in }\Omega_e .
		\end{cases}
	\end{equation}
	To obtain the well-posedness of \eqref{eq: main equation}, let us recall the eigenvalue problem of the logarithmic Laplacian.
	
	By \cite[Theorem 1.2]{CW2019dirichlet}, it is known that $\loglap$ in $\Omega$ has eigenvalues
	\begin{equation}\label{eigenvalues of the log Laplacian}
		\lambda_1(\Omega) <\lambda_2(\Omega) \leq \ldots \leq \lambda_k(\Omega) \leq \ldots  \nearrow \infty
	\end{equation}
	in a bounded domain $\Omega$. Any of these eigenvalues $\lambda_k(\Omega)$ may be positive, zero, or negative, depending on the size and shape of $\Omega$. In order to have a well-posed forward problem \eqref{eq: main equation}, let us impose the condition 
	\begin{equation}\label{eigenvalue}
		\lambda_1(\Omega)+q(x)\geq \lambda_0>0, \text{ for all }x\in \Omega,
	\end{equation}
	for some positive constant $\lambda_0$, where $\lambda_1(\Omega)$ is the first Dirichlet eigenvalue of $\loglap$ in $\Omega$. Lower bounds for $\lambda_1(\Omega)$ in terms of $|\Omega|$, the measure of $\Omega$, are given in \cite{CW2019dirichlet,laptev-weth}, see in particular \cite[Corollary 4.3 and Theorem 4.4]{laptev-weth}\footnote{Note that $\frac{1}{2}\loglap$ is considered in place of $\loglap$ in \cite{laptev-weth}, so the bounds there apply to $\frac{\lambda_1(\Omega)}{2}$ in place of $\lambda_1(\Omega)$.}.

	With the eigenvalue condition \eqref{eigenvalue} at hand, we can prove \eqref{eq: main equation} is well-posed for any $f$ contained in the associated {\em trace space} $\H_T(\Omega_e)$ which is defined in Section \ref{sec: forward problem} below. Hence, the DN map of \eqref{eq: main equation} can be defined as a map
	$$
	\Lambda_{q} \vcentcolon \H_T(\Omega_e) \mapsto \H_T(\Omega_e)^*.
	$$
	In particular, if $W_1,W_2 \Subset \Omega_e$ are open bounded subsets, then for every $f \in C_c^\infty(W_1)$ and $g \in C_c^\infty(W_1)$ we have
	$$
	\left\langle \Lambda_q f, g \right\rangle = \int_{W_2} \bigl(\loglap u_f\bigr) g\,dx  = 2 \int_{\R^n}(\log |\xi|)\widehat{u_f}(\xi)\overline{\widehat g(\xi)}\,d\xi,
	$$
	where $u_f$ is the unique (weak) solution to \eqref{eq: main equation}, and $\langle \cdot ,\cdot \rangle $ denotes the natural duality pairing between $\mathbb{H}_T(\Omega_e)$ and $\mathbb{H}_T(\Omega_e)^\ast$, see Section \ref{sec: forward problem} below. The key question studied in this paper is the following.
	
	\begin{enumerate}[\textbf{(IP)}]
		\item \label{IP}\textbf{Inverse Problem.} Can one determine $q$ via the DN map $\Lambda_q$?
	\end{enumerate}
	
	The following main result gives an affirmative answer to \ref{IP}.
	
	\begin{theorem}[Global uniqueness]\label{thm: uniqueness}
		Let $\Omega\subset \R^n$ be a bounded Lipschitz domain for $n\in \N$, and $W_1,W_2\Subset \Omega_e$ be nonempty bounded open sets. Assume that $q_j \in L^\infty(\Omega)$ satisfies \eqref{eigenvalue} for $j=1,2$. Let $\Lambda_{q_j}$ be the DN map of 
		\begin{equation}\label{eq: main equation j=12}
			\begin{cases}
				\LC \log (-\Delta) +q_j \RC u_j =0 &\text{ in }\Omega, \\
				u_j =f &\text{ in }\Omega_e,
			\end{cases}
		\end{equation}
		for $j=1,2$.
		Suppose that 
		\begin{equation}\label{eq: same DN map in thm 1}
			\langle \Lambda_{q_1} f, g \rangle =         \langle \Lambda_{q_2} f, g \rangle \quad \text{for any $f\in  C_c^{\infty}(W_1)$, $g \in C_c^\infty(W_2)$.}
		\end{equation}
		Then there holds $q_1=q_2$ in $\Omega$.
	\end{theorem}
	
	The proof of the preceding theorem is based on the UCP and Runge approximation for the logarithmic Schr\"odinger equation \eqref{eq: main equation}.\\
	
	On top of that, we also have a constructive uniqueness for $q$ using a \emph{monotonicity test}. Indeed, similarly as in the case of the fractional problem considered in \cite{harrach2017nonlocal-monotonicity}, one can derive an if-and-only-if monotonicity relation between bounded potentials and associated DN maps. More precisely, we have the following.
	\begin{theorem}[If-and-only-if monotonicity]
		\label{thm: equivalent monotonicity}
		Let $\Omega\subset \R^n$ be a bounded Lipschitz domain, and assume that $q_j \in L^\infty(\Omega)$ satisfies \eqref{eigenvalue} for $j=1,2$. Then the following are equivalent:
		\begin{enumerate}[(i)]
			\item $q_2 \ge q_1$ a.e. in $\Omega$;
			\item $\left\langle (\Lambda_{q_2}-\Lambda_{q_1})f,f \right\rangle \ge 0$ for all $f \in \H_T(\Omega_e)$;
			\item There exists a nonempty bounded open set $W\Subset \Omega_e$ with the property that
			\begin{equation}\label{quadratic sense}
				\left\langle \LC \Lambda_{q_1}-\Lambda_{q_2} \RC f, f \right\rangle \geq 0, \qquad \text{for all $f\in C^\infty_c(W).$}
			\end{equation}
		\end{enumerate}
	\end{theorem}

	In the following, if $W\Subset \Omega_e$ is a nonempty bounded open set, we say that
	\begin{equation}
		\label{eq:definition-quadratic-sense}
		\text{$\Lambda_{q_1}\leq\Lambda_{q_2}\;$ in $\;W$}  
	\end{equation}
	if (\ref{quadratic sense}) holds.
	
	\begin{remark}
		In the special case $W_1 = W_2=W$, Theorem \ref{thm: uniqueness} can be viewed as a corollary of Theorem~\ref{thm: equivalent monotonicity} since in this case condition (\ref{eq: same DN map in thm 1}) implies that both $\Lambda_{q_1}\leq \Lambda_{q_2}$ and $\Lambda_{q_2}\leq \Lambda_{q_1}$ hold in $\Omega$ and therefore $q_1 =q_2$ in $\Omega$ by Theorem~\ref{thm: equivalent monotonicity}. 
	\end{remark}
	
	As indicated above, the if-and-only-if monotonicity relation given by Theorem~\ref{thm: equivalent monotonicity} yields a constructive global uniqueness result in the case where $\lambda_1(\Omega) >0$ and $q \in L^\infty(\Omega)$ is nonnegative. For this purpose, let us recall that a point $x\in \R^n$ is called (Lebesgue) \emph{density one} for a measurable set $E$ if 
	\begin{align}
		\lim_{r\to 0}\dfrac{\left|B_r(x)\cap E\right|}{\left|B_r(x)\right|}=1,
	\end{align}
	where $B_r(x)$ denotes the ball of radius $r$ and centered at $x$. The space of \emph{density one simple functions} can be defined by 
	\begin{align*}
		\Sigma := \bigg\{ \varphi=\sum_{j=1}^m a_j \chi_{E_j}:\,  a_j\in \R,\ \text{$E_j\subseteq \Omega$ is a density one set} \bigg\},
	\end{align*}
	where we call a subset $E\subseteq \Omega$ a \emph{density one set} provided that $E$ is nonempty, measurable and has density one for all $x\in E$. It is not hard to find that density one sets have a positive measure, and finite intersections of density one sets are density one. Let, moreover, $\Sigma_{+,0}\subseteq \Sigma$ denote the subset of {\em nonnegative} density one simple functions.

	\begin{theorem}[Constructive uniqueness]\label{thm: const. uniqueness}
		Let $\Omega\subset \R^n$ be a bounded Lipschitz domain with $\lambda_1(\Omega)>0$, and let $q \in L^\infty(\Omega)$ be nonnegative. Moreover, let $W\Subset \Omega_e$ be an arbitrary nonempty bounded open set. Then the potential $q$ can be determined uniquely via the formula
		\begin{equation}
			q(x)=\sup \left\{ \varphi(x)\:: \: \text{$\varphi \in \Sigma_{+,0}$ and $\Lambda_\varphi \leq \Lambda_q$ in $W$} \right\}, \text{ for a.e. }x\in \Omega,
		\end{equation}
		where the relation $\Lambda_\varphi \leq \Lambda_q$ in $W$ is defined as above. 
	\end{theorem}
	
	Note that the proofs of both Theorems \ref{thm: uniqueness} and \ref{thm: const. uniqueness} rely on the Runge approximation and its applications for the logarithmic Schr\"odinger equation.\\
	
	\noindent \textbf{Organization of this article.} In Section \ref{sec: preliminaries}, we introduce several basic function spaces and a rigorous definition of the logarithmic Laplacian. In Section \ref{sec: forward problem}, we demonstrate the well-posedness of the exterior value problem \eqref{eq: main equation} for suitable function spaces. The existence of the DN map can be guaranteed by the well-posedness of \eqref{eq: main equation}. We prove Theorem \ref{thm: uniqueness} in Section \ref{sec: pf of global unique}, by using suitable integral identity and the Runge approximation. Finally, in Section \ref{sec: mono}, we derive the monotonicity relations between the DN maps and nonnegative potentials as given in Theorem~\ref{thm: equivalent monotonicity}, which will be applied to show the constructive uniqueness as stated in Theorem \ref{thm: const. uniqueness}.

	\section{Preliminaries}\label{sec: preliminaries}
	
	In this section, let us introduce and recall several useful properties for our study. Our notation for the Fourier transform is 
	\begin{equation}
		\widehat{f}(\xi) =\mathcal{F}f(\xi)=\int_{\R^n} e^{-\mathrm{i}x\cdot \xi} f(x)\, dx,
	\end{equation}
	where $\mathrm{i}=\sqrt{-1}$ denotes the imaginary unit. In what follows, let us always consider $u:\R^n\to \R$ as a (Lebesgue) measurable function, and $\Omega \subset \R^n$ be a bounded domain with Lipschitz boundary. 
	Given $s\in (0,1)$, the fractional Sobolev space $H^s(\R^n)$ is the standard $L^2$-based Sobolev space with the norm $\norm{u}_{H^s(\R^n)}=\big\| \LC 1+\abs{\xi}\RC^{s/2} \widehat{u} \big\|_{L^2(\R^n)}$.
	
	\subsection{The logarithmic Laplacian and associated function spaces}
	Given $s\in (0,1)$, recall that the fractional Laplacian $(-\Delta)^s$ can be defined via the Fourier transform 
	\begin{equation}
		(-\Delta)^s u \vcentcolon = \mathcal{F}^{-1} \big( \abs{\xi}^{2s}\widehat{u}(\xi)\big), \quad \text{for} \quad u \in \mathcal{S}(\R^n),
	\end{equation}
	where $\mathcal{S}(\R^n)$ stands for the Schwartz space. Alternatively, the fractional Laplacian can be written as a hypersingular integral operator of the form
	\begin{equation}\label{eq: integral representation of fractional Laplacian}
		(-\Delta)^s u (x) =\mathrm{P.V.}\int_{\R^n}\LC u(x)-u(z)\RC \mathcal{K}_s(x-z)\, dz
	\end{equation} 
	with the symmetric kernel
	\begin{equation}\label{eq: kernel for fractional Laplacian}
		\mathcal{K}_s(z)\vcentcolon = \frac{ C_{n,s} }{\abs{z}^{n+2s}},
	\end{equation}
	and
	\begin{equation}\label{constant C_n,s}
		C_{n,s}\vcentcolon= \frac{4^s\Gamma\LC \frac{n}{2}+s\RC }{\pi^{n/2}\abs{\Gamma(-s)}}.
	\end{equation} 
	The logarithmic Laplacian $\loglap$ appears in the study of $(-\Delta)^s$ in the limit $s \to 0$. Given $u\in C^\alpha_c(\R^n)$ for some $\alpha>0$, $\loglap u(x)$ can be uniquely defined by the asymptotic expansion\footnote{Here $o(s)$ denotes the small o, which satisfies $\frac{o(s)}{s}\to 0$ as $s\to 0^+$.}
	\begin{equation}
		(-\Delta)^s u(x)= u(x)+s \loglap u (x) +o(s), \quad \text{as} \quad s\to 0^+,
	\end{equation}
	so $\loglap u$ appears as a linear correction term in the second limit of \eqref{eq: limit of fra-Lap}. Formally, we thus have 
	\begin{equation}
		\frac{d}{ds} \Big|_{s=0} (-\Delta)^s u(x)=\loglap u(x),
	\end{equation}
	and $\loglap u$ has the symbol $2\log\abs{\xi}$ given by \eqref{eq: def log-Lap Fourier}.
	
	Moreover, from \cite[Theorem 1.1]{CW2019dirichlet} again, it is known that the logarithmic Laplacian admits an integral representation 
	\begin{equation}\label{eq: integral representation}
		\begin{split}
			\loglap u(x)= c_n \mathrm{P.V.}\int_{B_1(x)}\frac{u(x)-u(z)}{\abs{x-z}^n}\, dz - c_n \int_{\R^n \setminus B_1(x)} \frac{u(z)}{\abs{x-z}^n} \, dz +\rho_n u(x), 
		\end{split}
	\end{equation}
	for $x\in \R^n$ and $u\in C^\alpha_c(\R^n)$, for some $\alpha>0$, where $B_r(x)$ is the ball center at $x$ with radius $r>0$, and 
	\begin{equation}\label{eq: constant c_n and rho_n}
		c_n \vcentcolon = \pi^{-n/2}\Gamma\LC \frac{n}{2} \RC =\frac{2}{\left| \mathbb S^{n-1}\right|} , \quad  \rho_n\vcentcolon = 2\log 2 +\psi\LC \frac{n}{2}\RC -\gamma.
	\end{equation}	
	Here $\left| \mathbb S^{n-1}\right|$, $\gamma \vcentcolon= -\Gamma'(1)$ and $\psi =\frac{\Gamma'}{\Gamma}$ are the $(n-1)$-dimensional volume of the unit sphere in $\R^n$,  the Euler Mascheroni constant, and the Digamma function, respectively.

	In the distributional sense, the logarithmic Laplacian is defined on the function space 
	\begin{equation}\label{L^1_0(R^n)}
		L^1_0(\R^n) \vcentcolon =  \left\{ u\in L^1_{\mathrm{loc}}(\R^n) \vcentcolon \, \int_{\R^n}\frac{\abs{u(x)}}{(1+\abs{x})^n }\, dx <\infty\right\}.
	\end{equation}
	More precisely, for a function $u \in L^1_0(\R^n)$, the (distributional) logarithmic Laplacian is well-defined  by setting
	$$
	\LC \loglap u\RC (\phi) = \int_{\R^n} u\LC \loglap\phi\RC dx, \quad \text{for $\phi \in C^\infty_c(\R^n)$.}  
	$$
	On $\R^n$, the natural energy space associated with the logarithmic Laplacian is defined by 
	$$
	\mathbb{H}(\R^n)= \left\{ u \in L^2(\R^n)\::\:\int_{\R^n}\bigl|\log|\xi|\bigr| |\widehat u(\xi)|^2\,d\xi  < \infty\right\}.
	$$
	It is a Hilbert space with a scalar product
	\begin{equation}\label{eq: inner product R n}
		(v,w) \mapsto		\langle v,w \rangle_{\mathbb{H}(\R^n)} \vcentcolon= \langle v,w \rangle_{L^2(\R^n)} + \int_{\R^n} \bigl|\log|\xi|\bigr| \widehat v(\xi) \overline{\widehat w(\xi)} \, d \xi
	\end{equation}
	and induced norm $\norm{v}_{\mathbb{H}(\R^n)}=\sqrt{\langle v,v \rangle _{\mathbb{H}(\R^n)}}$. 
	
	Next, the bilinear form associated with the logarithmic Laplacian is given by 
	\begin{equation}\label{eq:def-bilinear-form-loglap}
		\begin{split}
			B_0 &: \mathbb{H}(\R^n) \times \mathbb{H}(\R^n) \to \R, \\
			B_0(v,w)&:= 2\int_{\R^n} \log |\xi| \widehat v(\xi) \overline{\widehat w(\xi)}d\xi.
		\end{split} 
	\end{equation}
	Indeed this bilinear form is well-defined and continuous on $\H(\R^n)$ since 
	\begin{equation}\label{B-0-upper-bound}
		\begin{split}
			\int_{\R^n} \bigl|\log |\xi| \bigr| |\widehat v(\xi)| |\overline{\widehat w(\xi)}|d\xi &\le  \Bigl(\int_{\R^n} \bigl|\log |\xi|\bigr||\widehat{v}(\xi)|^2\,\xi\Bigr)^{1/2}\Bigl(\int_{\R^n} \bigl|\log |\xi|\bigr||\widehat{v}(\xi)|^2\,\xi\Bigr)^{1/2} \\
			&\le \|u\|_{\H(\R^n)}\|v\|_{\H(\R^n)} \quad \text{for $u,v \in \H(\R^n)$.}
		\end{split}
	\end{equation}
	The bilinear form $B_0$ allows to define $\loglap$ in the weak sense as an operator 
	$$
	\loglap: \mathbb{H}(\R^n) \to \mathbb{H}(\R^n)^\ast
	$$
	by
	$$
	\left \langle \loglap v, w \right\rangle = B_0 (v,w)
	$$
	for $v,w \in \mathbb{H}(\R^n)$,  where $\left\langle \cdot, \cdot \right\rangle $ denotes the duality pairing between $\mathbb{H} (\R^n)^\ast$ and $\mathbb{H}(\R^n)$. 
	By the definition of \eqref{eq:def-bilinear-form-loglap}, we have the symmetry property 
	\begin{equation}\label{eq: symmetry log}
		B_0(v,w)=B_0(w,v), \quad \text{for $v,w \in \mathbb{H}(\R^n)$.}
	\end{equation}
	Finally, for an open subset $\Omega \subset \R^n$, we consider the closed subspace 
	\begin{equation}\label{eq: H(Omega) space}
		\begin{split}
			\mathbb{H}_0(\Omega) \vcentcolon = \{ u \in \mathbb{H}(\R^n): \, u=0 \text{ in }\Omega_e \}
		\end{split}
	\end{equation}
	of $\mathbb{H}(\R^n)$, which is again a Hilbert space. We note the following observation.
	
	\begin{lemma}
		\label{equivalent-norms}
		Let $\Omega \subset  \R^n$ be an open set of finite measure.  Then the inner product 
		\begin{equation}\label{eq: inner product}
			(v,w) \mapsto \langle v,w \rangle _{\mathbb{H}_0(\Omega)} \vcentcolon= \int_{\abs{x-z}\leq 1} \frac{\LC v(x)-v(z)\RC \LC w(x)-w(z)\RC }{\abs{x-z}^n}\, dxdz
		\end{equation}
		induces an equivalent norm 
		$$v \mapsto \norm{v}_{\mathbb{H}_0(\Omega)}=\sqrt{\langle v,v \rangle _{\mathbb{H}_0(\Omega)} } \text{ on } \mathbb{H}_0(\Omega).
		$$
		Moreover, $\mathbb{H}_0(\Omega)$ is compactly embedded into $L^2(\Omega)$. 
	\end{lemma}
	
	\begin{proof}
		In the following, the letter $C>0$ stands for different positive constants.  We shall use the symbol $\xi \mapsto \log (1 + |\xi|^2)$, associated with the logarithmic Schrödinger operator $\log(-\Delta + 1)$, as a comparison function. The operator $\log(-\Delta + 1)$ is a singular integral operator with kernel function
		$$
		z \mapsto \ell(z)=  C_n K_{n/2}(|z|)\abs{z}^{-n/2},
		$$
		for some constant $C_n$ depending only on $n\in \N$, where $K_{\nu}$ denotes the modified Bessel function of second kind with index $\nu$ (see e.g. \cite[Section 1]{feulefack}). As a consequence, for every
		$\phi \in L^2(\R^n)$ we have
		\begin{equation}
			\label{eq:log-schroedinger-comparison}
			\int_{\R^n}\log(1+|\xi|^2)|\widehat \phi(\xi)|^2\,d\xi = \int_{\R^n \times \R^n} \ell(|x-z|)\LC \phi(x)-\phi(z)\RC^2 \, dxdz, 
		\end{equation}
		where the finiteness of one side of this equality implies the finiteness of the other. We also note that, by the positivity of $K_{\nu}$ and its asymptotics at $z=0$ (see e.g. \cite[Section 1]{feulefack} again), we have
		$$
		\frac{\ell(z)}{C_0}  \le |z|^{-n} \le C_0 \ell(z), \quad \text{for $z \in B_1(0) \setminus \{0\}$ with a constant $C_0>0$.}
		$$
		Therefore, for every $\phi \in \mathbb{H}_0(\Omega) \subset \mathbb{H}(\R^n)$, we now have the estimate
		\begin{equation}
			\begin{split}
				\int_{\abs{x-z}\leq 1} \frac{\LC \phi(x)-\phi(z) \RC^2}{\abs{x-z}^n}\, dxdz 
				&\le C \int_{\R^n \times \R^n} \ell(|x-z|) \LC \phi(x)-\phi(z)\RC^2\, dxdz\\
				&= C \int_{\R^n}\log(1+|\xi|^2)|\widehat \phi(\xi)|^2\,d\xi\\
				&\le C \int_{\R^n} (1+\bigl|\log|\xi|\bigr|)|\widehat \phi(\xi)|^2\,d\xi\\
				&= C \|\phi\|_{\mathbb{H}(\R^n)}^2. 
			\end{split}
		\end{equation}
		Moreover, by (\ref{eq:log-schroedinger-comparison}) we have 
		\begin{equation*}
			\begin{split}
				\|\phi\|_{\mathbb{H}(\R^n)}^2  &= \int_{\R^n} (1+ \abs{\log|\xi|})|\widehat \phi(\xi)|^2\,d\xi \\
				&\le \int_{\R^n}\bigl(1 - 1_{B_1(0)}(\xi)\log |\xi| + \log(1+|\xi|^2)\bigr)|\widehat \phi(\xi)|^2\,d\xi\\
				&\le \|\widehat \phi\|_{L^2(\R^n)}^2+ \|\widehat \phi\|_{L^\infty(B_1(0)}^2 \int_{B_1(0)}(-\log |\xi|)d\xi\\
				&\quad \, +C \int_{\R^n \times \R^n} \ell(|x-z|) \LC \phi(x)-\phi(z) \RC^2\, dxdz\\
				&\le  \|\phi\|_{L^2(\Omega)}^2+ C \|\phi\|_{L^1(\Omega)}^2
				+C \int_{\abs{x-z}\leq 1} \frac{\LC \phi(x)-\phi(z)\RC^2}{|x-y|^n}\, dxdz\\
				&\quad \, + C \int_{\abs{x-z}> 1}\ell(x-z) \LC \phi^2(x)+\phi^2(z)\RC \, dxdz\\
				&\le C \Bigl(\|\phi\|_{L^2(\Omega)}^2 + \|\phi\|_{\mathbb{H}_0(\Omega)}^2 + \|\phi\|_{L^2(\R^n)}^2\int_{\R^n \setminus B_1(0)}
				\ell(y)\,dy \Bigr) \\
				&\le C \bigl(\|\phi\|_{L^2(\Omega)}^2 + \|\phi\|_{\mathbb{H}_0(\Omega)}^2\bigr),
			\end{split}
		\end{equation*}
		where $1_{B_1(0)}(\xi)=\begin{cases}
			1 &\text{ for }\xi \in B_1(0)\\
			0 &\text{ otherwise}
		\end{cases}$ is the characteristic function.
		Here we used the assumption that $\Omega$ has finite measure $\abs{\Omega}\in (0,\infty)$, such that the H\"older inequality 
		$$
		\|\phi\|_{L^1(\Omega)} \le |\Omega|^{1/2} \|\phi\|_{L^2(\Omega)}
		$$
		holds.
		Combining the above estimates with the obvious bound
		$$
		\|\phi\|_{L^2(\Omega)} = \|\phi\|_{L^2(\R^n)}= \|{\widehat \phi}\,\|_{L^2(\R^n)} \le \|\phi\|_{\mathbb{H}(\R^n)} \quad \text{for $\phi \in \mathbb{H}_0(\Omega)$}, 
		$$
		we deduce that
		$$
		\phi \mapsto \|\phi\|_* := \bigl(\|\phi\|_{L^2(\Omega)}^2 + \|\phi\|_{\mathbb{H}_0(\Omega)}^2\bigr)^{\frac{1}{2}}
		$$
		is an equivalent norm to $\|\cdot \|_{\mathbb{H}(\R^n)}$ on $\mathbb{H}_0(\Omega)$. Moreover, by \cite[Theorem 1.2]{jarohs-weth}, the embedding $(\mathbb{H}_0(\Omega),\|\cdot\|_*) \hookrightarrow L^2(\Omega)$ is compact. From this compactness and the fact that $\|\phi\|_{\mathbb{H}_0(\Omega)}>0$ for all $\phi \in \mathbb{H}_0(\Omega) \setminus \{0\}$, it follows by a standard argument that $\|\phi \|_{L^2(\Omega)} \le C \|\phi\|_{\mathbb{H}_0(\Omega)}$ for all $\phi \in \mathbb{H}_0(\Omega)$, so the norms $\|\cdot\|_*$ and $\|\cdot\|_{\mathbb{H}_0(\Omega)}$ are equivalent on $H^1_0(\Omega)$. This proves the result.
	\end{proof}	
	
	We also have the following useful estimate.

	\begin{lemma}
		\label{first-useful-est}
		For $\alpha>0$, we have $C^\alpha_c(\R^n) \subset \H(\R^n)$. Moreover, for every nonempty bounded open set $W \subset \R^n$, there exists a constant $C=C(W,\alpha)>0$ with 
		\begin{equation}\label{eq: estimate log f}
			\norm{f}_{\mathbb{H}(\R^n)} \leq C \norm{f}_{C^{\alpha}(W)} \qquad \text{for all $f \in C^\alpha_c(W)$.}
		\end{equation}
	\end{lemma}

	\begin{proof}
		It clearly suffices to prove (\ref{eq: estimate log f}). Since $W \subset B_R(0)$ for $R$ chosen sufficiently large, we may assume without loss of generality that $W= B_R(0)$ for some $R>1$. 
		Since $\norm{\cdot}_{\mathbb{H}(\R^n)}$ is equivalent to $\norm{\cdot}_{\mathbb{H}_0(W)}$ on $\mathbb{H}_0(W)$, it suffices, by approximation, to prove the estimate with $\norm{\cdot}_{\mathbb{H}_0(W)}$ in place of $\norm{\cdot}_{\mathbb{H}(\R^n)}$.
		
		Let $f \in C^\alpha_c(W)$, then we can write,  similar to \cite[Proposition 3.2]{CW2019dirichlet}, 
		\begin{equation}\label{first-useful-est-proof-0}
			\begin{split}
				\norm{f}_{\mathbb{H}_0(W)}^2 = \int_{\stackrel{x,z \in W}{\abs{x-z}<1}} \frac{(f(x)-f(z))^2}{|x-z|^n}\,dx dz
				+2 \int_{W}|f(x)|^2 \kappa_{W}(x)\,dx
			\end{split}
		\end{equation}
		with
		\begin{equation}
			\label{eq:kappa-omega}
			\kappa_{W}(x)= \int_{(\R^n\setminus \overline{W}) \cap B_1(x)}|y-x|^{-n}\, dx.
		\end{equation}
		It is easy to compute that
		$$
		\kappa_{W}(x) \le \left|\mathbb{S}^{n-1}\right| \Bigl[\log \frac{1}{\dist(x,\partial \Omega)}\Bigr]_+ \le
		\left|\mathbb{S}^{n-1}\right| \Bigl[\log \frac{1}{R-|x|}\Bigr]_+,  
		$$
		for  $x \in W= B_R(0)$, which implies that
		\begin{equation}\label{first-useful-est-proof-1}
			\begin{split}
				\int_{W}|f(x)|^2 \kappa_{W}(x)\,dx & \le |\mathbb S^{n-1}|\|f\|_{L^\infty(W)}^2\int_{B_R(0)} \Bigl[\log \frac{1}{R-|x|}\Bigr]_+dx \\
				&\le
				|\mathbb S^{n-1}|^2 \|f\|_{C^\alpha(W)}^2\int_{R-1}^{R} r^{n-1}\log \frac{1}{R-r}\, dr  \\
				&\le C\|f\|_{C^\alpha(W)}^2,
			\end{split}	
		\end{equation}
		with a constant $C=C(W)>0$. Moreover, we have
		\begin{equation}\label{first-useful-est-proof-2}
			\begin{split}
				\int_{\stackrel{x,z \in W}{\abs{x-z}<1}} \frac{(f(x)-f(z))^2}{|x-z|^n}\, dxdz &\le \|f\|_{C^\alpha(W)}^2 \int_{\stackrel{x,z \in W}{\abs{x-z}<1}}|x-z|^{\alpha-n} \, dxdz \\
				&\le |W| \|f\|_{C^\alpha(W)}^2 \int_{B_1(0)}|y|^{\alpha-n}\, dy \\
				&\le \frac{|W||\mathbb{S}^{n-1}|}{\alpha} \|f\|_{C^\alpha(W)}^2.
			\end{split}
		\end{equation}
		Combining~\eqref{first-useful-est-proof-0}, \eqref{first-useful-est-proof-1} and \eqref{first-useful-est-proof-2}, we obtain \eqref{eq: estimate log f} with a constant $C=C(W,\alpha)>0$. This proves the assertion.
	\end{proof}

	Throughout the remainder of this paper, let us denote the $L^2$-inner product by
	\begin{equation}
		\LC \phi, \psi\RC_{L^2(A)}\vcentcolon = 	\int_{A} \phi \psi \, dx,
	\end{equation}
	for any $\phi,\psi\in L^2(A)$ and for any $A\subset \R^n$.
	
	\subsection{The energy form of the logarithmic Schrödinger type operator}

	Let $\Omega \subset \R^n$ be a bounded open set with Lipschitz boundary and $q \in L^\infty(\Omega)$. We consider the bilinear form associated with the operator $\loglap +q$ given by
	\begin{equation}
		\label{eq:def-bilinear-form-loglap-q}
		\begin{split}
			B_q & : \mathbb{H}(\R^n) \times \mathbb{H}(\R^n) \to \R, \\
			B_q(v,w)&:=B_0(v,w)+ \LC q v , w \RC_{L^2(\Omega)},
		\end{split}
	\end{equation}
	where $B_0$ is defined in (\ref{eq:def-bilinear-form-loglap}). We recall that
	the variational characterization of $\lambda_1(\Omega)$, the first Dirichlet-eigenvalue of $\loglap$ on $\Omega$, is then given by
	\begin{equation}
		\label{eq:var-char-lambda-1}
		\lambda_1(\Omega) = \inf_{\stackrel{u \in \H_0(\Omega)}{u \not = 0}}\frac{B_0(u,u)}{\|u\|_{L^2(\Omega)}^2}.  
	\end{equation}
	Hence the eigenvalue condition, if satisfied, implies that 
	\begin{equation}
		\label{eq:eigenvalue-reformulated}
		B_q(u,u) \ge \lambda_0 \|u\|_{L^2(\Omega)}^2 \qquad \text{for all $u \in \H_0(\Omega)$.}  
	\end{equation}
	We also note the following useful estimates.
	
	\begin{lemma}
		\label{B-q-estimates}
		Let $B_q(\cdot, \cdot)$ be the bilinear form given by \eqref{eq:def-bilinear-form-loglap-q}, then we have
		\begin{equation}
			\label{eq:B-q-est-1}
			|B_q(u,v)| \le (2+\|q\|_{L^\infty(\Omega)})  \|u\|_{\H(\R^n)}\|v\|_{\H(\R^n)} \quad \text{for $u,v \in \H(\R^n)$.}
		\end{equation}
		and
		\begin{equation}
			\label{eq:B-q-est-2}
			B_q(u,u) \ge 2 \|u\|_{\H(\R^n)}^2 - C \|u\|_{L^2(\Omega)}^2 \qquad \text{for $u \in \H_0(\Omega)$}
		\end{equation}
		with a constant $C>0$.  If moreover (\ref{eigenvalue}) holds, then 
		\begin{equation}
			\label{eq:B-q-est-3}
			B_q(u,u) \ge C \|u\|_{\H(\R^n)}^2 \qquad \text{for $u \in \H_0(\Omega)$} 
		\end{equation}
		with a constant $C>0$.
	\end{lemma}

	\begin{proof}
		For $u,v \in \H(\R^n)$ we have, by (\ref{B-0-upper-bound}),
		\begin{align*}
			|B_q(u,v)| &\le |B_0(u,v)| + |\LC q v , w \RC_{L^2(\Omega)}| \\
			&\le 2\int_{\R^n}|\log |\xi|| |\widehat{u}(\xi)| |\widehat{ v}(\xi)|d\xi + \|q\|_{L^\infty(\Omega)}\|u\|_{L^2(\Omega)}\|v\|_{L^2(\Omega)}\\
			&\le 2 \|u\|_{\H(\R^n)}\|v\|_{\H(\R^n)} + \|q\|_{L^\infty(\Omega)}\|u\|_{L^2(\R^n)}\|v\|_{L^2(\R^n)}\\
			&\le (2 + \|q\|_{L^\infty(\Omega)})\|u\|_{\H(\R^n)}\|v\|_{\H(\R^n)}, 
		\end{align*}
		as claimed in (\ref{eq:B-q-est-1}). Moreover, we have
		\begin{align*}
			B_0(u,u) &= 2\int_{\R^n}\log|\xi| |\widehat{u}(\xi)|^2\,d\xi \\
			&= 2 \int_{\R^n}(1+ \bigl| \log|\xi|\bigr|) |\widehat{u}(\xi)|^2\,d\xi  -2 \int_{\R^n}|\widehat{u}(\xi)|^2\,d\xi  \\
			&\quad \, +  4 \int_{B_1(0)}(\log|\xi|)|\widehat{u}(\xi)|^2\,d\xi\\
			&\geq  2\|u\|_{\H(\R^n)}^2 -2\|u\|_{L^2(\Omega)}^2 +4\|\widehat{u}\|_{L^\infty(\R^n)}^2 \int_{B_1(0)} \log|\xi| \, d\xi\\
			&\ge 2\|u\|_{\H(\R^n)}^2 -2\|u\|_{L^2(\Omega)}^2 -C \|u\|_{L^1(\Omega)}^2\\
			&\ge 2\|u\|_{\H(\R^n)}^2 -C \|u\|_{L^2(\Omega)}^2, \quad \text{for $u \in \H_0(\Omega)$.} 
		\end{align*}
		Here we used again the fact that $\|u\|_{L^1(\Omega)} \le C \|u\|_{L^2(\Omega)}$ thanks to $\Omega$ has finite measure. We conclude that 
		$$
		B_q(u,u) \ge  B_0(u,u) -\|q\|_{L^\infty(\Omega)}\|u\|_{L^2(\Omega)}^2
		\ge 2 \|u\|_{\H(\R^n)}^2 - C \|u\|_{L^2(\Omega)}^2,
		$$
		for $u \in \H_0(\Omega)$, as claimed in (\ref{eq:B-q-est-2}).
		
		Finally, if (\ref{eigenvalue}) holds, then by (\ref{eq:eigenvalue-reformulated}) and (\ref{eq:B-q-est-2}) we have, for $u \in \H_0(\Omega)$ and $\eps \in (0,1)$,
		\begin{align*}
			B_q(u,u) &= (1-\eps)B_q(u,u) + \eps B_q(u,u)\\
			&\ge (1-\eps)\lambda_0\|u\|_{L^2(\Omega)}^2 + \eps \big(2\|u\|_{\H(\R^n)}^2 - C \|u\|_{L^2(\Omega)}^2\big)\\
			&= 2 \eps \|u\|_{\H(\R^n)}^2 + \big((1-\eps)\lambda_0 - C\eps\big)\|u\|_{L^2(\Omega)}^2.
		\end{align*}
		Since $\lambda_0>0$, we can then choose $\eps \in (0,1)$ such that $(1-\eps)\lambda_0 - C\eps \ge 0$. Then (\ref{eq:B-q-est-3}) holds with $C = 2 \eps$. This concludes the proof.
	\end{proof}
	
	\section{The forward problem}\label{sec: forward problem}

	Let $\Omega \subset \R^n$ be a bounded open set with Lipschitz boundary, and let $q \in L^\infty(\Omega)$ satisfy \eqref{eigenvalue}. In this section, we study the forward Dirichlet problem for the logarithmic Schr\"odinger type equation $\LC \loglap  + q \RC u = F$ in $\Omega$. For this, we need some preparations. We define the trace space
	$$
	\H_T(\Omega_e):= \big\{f\big|_{\Omega_e} \::\: f \in \H(\R^n)\big\}.
	$$
	To define a suitable norm on $\H_T(\Omega_e)$, we note the following result.
	\begin{lemma}
		\label{minimal extension}
		For every function $f \in \H_T(\Omega_e)$, the infimum 
		\begin{equation}
			\label{eq:infimum-energy-extension}
			c_f:= \inf \{\|g\|_{\H(\R^n)} \::\: g \in \H(\R^n), \: g\big|_{\Omega_e} = f\}
		\end{equation}
		admits a minimizer $\tilde f \in \H(\R^n)$. Moreover, $\tilde f$ is uniquely determined by the property that
		\begin{equation}
			\label{eq:unique-determination}
			\big\langle \tilde f, h \big\rangle_{\H(\R^n)} = 0 \qquad \text{for all $h \in \H_0(\Omega)$,}  
		\end{equation}
		and the map $f \mapsto \tilde f$ is linear. 
	\end{lemma}
	
	\begin{proof}
		Let $f \in \H_T(\Omega_e)$  By definition of $\H_T(\Omega_e)$, the set 
		$$
		M_f:= \big\{ g \in \H(\R^n)\,: \, g\big|_{\Omega_e} = f \big\}
		$$ 
		is nonempty. Let $\LC f_k \RC_{k\in \N}$ be a minimizing sequence in $M_f$ for the infimum in (\ref{eq:infimum-energy-extension}). Since $f_k$ is bounded in $\H(\R^n)$, we may pass to a subsequence such that $f_k \rightharpoonup \tilde f \in \H(\R^n)$ and therefore
		$$
		\|\tilde f\|_{\H(\R^n)} \le \lim_{k \to \infty}  \|f_k\|_{\H(\R^n)} = c_f.
		$$
		Moreover, we have $h_k:= f_k-f_1$ in $\H_0(\Omega)$ for all $k \in \N$, and
		$h_k \rightharpoonup \tilde f-f_1$. Since $\H_0(\Omega)\subset \H(\R^n)$ is a closed subspace, it follows that $\tilde f-f_1 \in \H_0(\Omega)$, hence $\tilde f \in M_f$ and $\tilde f$ is a minimizer for (\ref{eq:infimum-energy-extension}). It then follows that for any $h \in \H_0(\Omega)$ we have
		$$
		0 \le \lim_{t \to 0} \frac{1}{t}\bigl(\|\tilde f \pm t h \|_{\H(\R^n)}^2 - \|\tilde f\|_{\H(\R^n)}^2\bigr) = \pm 2   \langle \tilde f, h \rangle_{\H(\R^n)}
		$$
		and therefore (\ref{eq:unique-determination}) follows. Moreover, if $\tilde f_* \in M_f$ is another function satisfying (\ref{eq:unique-determination}), then $\tilde f - \tilde f_* \in \H_0(\Omega) \cap \bigl(\H_0(\Omega)\bigr)^\perp = \{0\}$ and therefore $\tilde f = \tilde f_*$. This shows that $\tilde f$ is uniquely determined by (\ref{eq:unique-determination}), and from this, the linearity of the map $f \to \tilde f$ follows. This proves the assertion
	\end{proof}
	
	Lemma~\ref{minimal extension} shows that a norm $\|\cdot\|_{\H_T(\Omega_e)}$ is well-defined by setting
	\begin{equation}
		\label{eq:def-trace-norm}
		\begin{split}
			\|f\|_{\H_T(\Omega_e)}: = \|\tilde f\|_{\H(\R^n)}=  \inf \{\|g\|_{\H(\R^n)} \::\: g \in \H(\R^n), \: g\big|_{\Omega_e} = f\},
		\end{split}
	\end{equation}
	for $f \in \H_T(\Omega_e)$. 
	
	\subsection{The Dirichlet problem}

	For $f \in \H_T(\Omega_e)$ and $F \in L^2(\Omega)$, we now consider the Dirichlet problem
	\begin{equation}\label{eq: well-posedness}
		\begin{cases}
			\LC \loglap +q \RC  u =F &\text{ in }\Omega, \\
			u=f &\text{ in }\Omega_e.
		\end{cases}
	\end{equation}
	To define the notion of weak solutions, we use the bilinear form $B_q$ defined in (\ref{eq:def-bilinear-form-loglap-q}).  
	
	\begin{definition}[Weak solutions]
		Let $\Omega \subset \R^n$ be a bounded open set with Lipschitz boundary. Given $f \in \mathbb{H}_T(\Omega_e)$ and $F \in L^2(\Omega)$, a function $u\in \mathbb{H}(\R^n)$ is called a weak solution to \eqref{eq: well-posedness} if $u \equiv f$ in $\Omega_e$ and 
		\begin{equation}
			B_q (u,\varphi) = \LC F, \varphi \RC_{L^2(\Omega)} ,\quad \text{for any $\varphi \in \mathbb{H}_0(\Omega)$.} 
		\end{equation}
	\end{definition}

	We then have the following well-posedness result. 
	
	\begin{lemma}[Well-posedness]\label{lem: well-posedness}
		Let $\Omega  \subset \R^n$ be a bounded open set with Lipschitz boundary, and $ q\in L^\infty(\Omega)$ satisfy \eqref{eigenvalue}. Then, for any $f \in \H_T(\Omega_e)$ and $F \in L^2(\Omega)$, there exists a unique weak solution $u\in \mathbb{H}(\R^n)$ to \eqref{eq: well-posedness}. In addition, there holds 
		\begin{equation}\label{eq: well-posed estimate}
			\norm{u}_{\mathbb{H}(\R^n)}\leq C \LC \norm{F}_{L^2(\Omega)} + \|f\|_{\H_T(\Omega_e)}\RC,
		\end{equation}
		for some constant $C>0$ independent of $u$, $F$, $f$.
	\end{lemma}

	\begin{proof}[Proof of Lemma \ref{lem: well-posedness}]
		Let $\tilde f$ be the $\|\cdot\|_{\H(\R^n)}$-minimizing extension of $f$ given by Lemma~\ref{minimal extension}, which satisfies $\|\tilde f\|_{\H(\R^n)} = \|f\|_{\H_T(\Omega_e)}$. Then
		$u\in \mathbb{H}(\R^n)$ is a weak solution to \eqref{eq: well-posedness} if and only if $v= u-\tilde f \in \mathbb{H}_0(\Omega)$, and $v$ satisfies
		\begin{equation}
			\label{weak-sol-transformed}
			B_q (v,\varphi) = \LC F, \varphi \RC_{L^2(\Omega)}-B_q(\tilde f,\varphi) ,\quad \text{for any $\varphi \in \mathbb{H}_0(\Omega)$.} 
		\end{equation}
		The existence of a unique $v \in \mathbb{H}_0(\Omega)$ with this property follows from the Lax-Milgram theorem, since, by Lemma~\ref{B-q-estimates}, $B_q$ is a continuous bilinear form on $\mathbb{H}_0(\Omega)$ satisfying (\ref{eq:B-q-est-3}). This shows the existence of a unique weak solution $u\in \mathbb{H}(\R^n)$ to \eqref{eq: well-posedness}, and $u = v + \tilde f$ with $v \in \mathbb{H}_0(\Omega)$ as above. Moreover, by (\ref{eq:B-q-est-3}) and (\ref{eq:B-q-est-1}) we have, with a constant $C>0$, 
		\begin{align*}
			\|v\|_{\H(\R^n)}^2 &\le C B_q(v,v)\\
			&= C\bigl( \LC F, v \RC_{L^2(\Omega)}-B_q(\tilde f,v)\bigr)\\
			&\le C \bigl(\|v\|_{L^2(\Omega)} \|F\|_{L^2(\Omega)} + \|\tilde f\|_{\H(\R^n)}\|v\|_{\H(\R^n)}\bigr)\\
			&\le C \|v\|_{\H(\R^n)} \bigl( \|F\|_{L^2(\Omega)} + \|\tilde f\|_{\H(\R^n)}\bigr)
		\end{align*}
		and therefore
		\begin{equation}
			\label{eq:v-estimate}
			\|v\|_{\H(\R^n)} \le C \bigl(\|F\|_{L^2(\Omega)} + \|\tilde f\|_{\H(\R^n)}\bigr).    \end{equation}
		Consequently, 
		\begin{align*}
			\|u\|_{\H(\R^n)} & \le \|v\|_{\H(\R^n)} + \|\tilde f\|_{\H(\R^n)} \\ 
			&\le C \bigl(\|F\|_{L^2(\Omega)} +\|\tilde f\|_{\H(\R^n)}\bigr)\\
			&= C \bigl(\|F\|_{L^2(\Omega)} +\|f\|_{\H_T(\Omega_e)}\bigr),
		\end{align*}
		as claimed.
	\end{proof}
	
	\begin{corollary}
		\label{cor: well-posedness}
		Let $\Omega \subset \R^n$ be a bounded open set with Lipschitz boundary, let $ q\in L^\infty(\Omega)$ satisfy \eqref{eigenvalue}, and let $W \subset \Omega_e$ be a nonempty bounded open set. Then for every $f \in C^\alpha_c(W)$, there exists a unique weak solution $u\in \mathbb{H}(\R^n)$ to \eqref{eq: well-posedness}. In addition, there holds 
		\begin{equation}\label{eq: well-posed estimate-corollary}
			\norm{u}_{\mathbb{H}(\R^n)}\leq C \LC \norm{F}_{L^2(\Omega)} + \|f\|_{C^\alpha(W)}\RC,
		\end{equation}
		for some constant $C>0$ independent of $u$, $F$, $f$.
	\end{corollary}
	
	\begin{proof}
		The result follows by combining Lemma~\ref{first-useful-est} and Lemma~\ref{lem: well-posedness}, since $f \in C^\alpha_c(W) \subset \H(\R^n)$ and
		$$
		\bigl\| f\big|_{\Omega_e} \bigr\|_{\H_T(\Omega_e)} \le \|f\|_{\H(\R^n)} \le C
		\|f\|_{C^\alpha(W)}
		$$
		by definition of $\|\cdot\|_{\H_T(\Omega_e)}$.
	\end{proof}

	\subsection{The DN map}
	
	With Theorem \ref{lem: well-posedness} at hand, the DN map of \eqref{eq: main equation} can be defined rigorously via the bilinear form $B_q$ defined in (\ref{eq:def-bilinear-form-loglap-q}).

	\begin{lemma}[The DN map]
		Let $\Omega \subset \R^n$ be a bounded Lipschitz domain, and let $q\in L^\infty(\Omega)$ fulfill \eqref{eigenvalue}. Define 
		\begin{equation}\label{eq: defi DN map}
			\begin{split}
				\left\langle  \Lambda_{q}f, g \right\rangle \vcentcolon = B_q(u_f,v_g),
			\end{split}
		\end{equation}
		for any $f, g \in \H_T(\Omega_e)$, where $u_f\in \mathbb{H}(\R^n)$ is the solution to \eqref{eq: main equation}, and $v_g\in \mathbb{H}(\R^n)$ can be any representative function with $\left. v_g\right|_{\Omega_e}=g$. Then 
		\begin{equation}\label{eq: mapping DN map}
			\Lambda_q \vcentcolon \H_T(\Omega_e)\to \H_T(\Omega_e)^*
		\end{equation}
		is a bounded linear operator, which satisfies the symmetry property
		\begin{equation}\label{eq: self-adjoint}
			\left\langle  \Lambda_{q}f, g \right\rangle = 	\left\langle  \Lambda_{q}g, f \right\rangle \qquad \text{for any $f, g \in \H_T(\Omega_e)$.}
		\end{equation}
	\end{lemma}

	\begin{proof}
		First, if there are two functions $v_g, \wt v_g \in \mathbb{H}(\R^n)$ such that $\left. v_g\right|_{\Omega_e}= \left.\wt v_g \right|_{\Omega_e}$, then $v_g-\wt v_ g \in \mathbb{H}_0(\Omega)$. Therefore, by the linearity of $B_q(u_f ,\cdot)$, one has 
		\begin{equation}
			B_q (u_f,  \wt v_{g}) = \underbrace{B_q (u_f,v_g) + B_q (u_f , \wt v_{g}-v_g)=B_q (u_f,v_g)}_{\text{since }u_f \text{ solves }\eqref{eq: main equation} \text{ and }\wt v_{g}-v_g \in \mathbb{H}_0(\Omega)} ,
		\end{equation}
		which shows \eqref{eq: defi DN map} is well-defined. To show the boundedness of $\Lambda_q$, we let $\tilde g \in \H(\R^n)$ be the $\|\cdot\|_{\H(\R^n)}$-minimizing extension of $g \in \H_T(\Omega_e)$ given by Lemma~\ref{minimal extension}, so that $\|\tilde g\|_{\H(\R^n)}= \|g\|_{\H_T(\Omega_e)}$. Then, by Lemma~\ref{B-q-estimates} and \eqref{eq: well-posed estimate} applied with $F=0$, we have 
		$$			\left|  \left\langle \Lambda_q f, g \right\rangle \right|  =\left| B_q (u_f,\tilde g)\right| \leq C\|u_f\|_{\H(\R^n)}\|\tilde g\|_{\H(\R^n)} \le C \|f\|_{\H_T(\Omega_e)}\|g\|_{\H_T(\Omega_e)}.
		$$
		This shows that $\Lambda_q \vcentcolon \H_T(\Omega_e)\to \H_T(\Omega_e)^*$ is bounded. Now, the symmetry of the DN map can be seen from the symmetry of the bilinear form $B_q$. This proves the assertion. 	
	\end{proof}
	
	We can also derive the integral identity.
	
	\begin{lemma}[Integral identity]\label{lem: integral id}
		Let $\Omega \subset \R^n$ be a bounded Lipschitz domain, and let $q_j\in L^\infty(\Omega)$ satisfy (\ref{eigenvalue}) for $j=1,2$. For any $f_1,f_2 \in \H_T(\Omega_e)$, we then have 
		\begin{equation}\label{eq: integral id}
			\left\langle \LC \Lambda_{q_1}-\Lambda_{q_2} \RC f_1, f_2 \right\rangle =\int_{\Omega} \LC q_1 -q_2 \RC u_{f_1}u_{f_2}\, dx,
		\end{equation}
		where $u_{f_j}$ is the unique weak solution to
		\begin{equation}
			\begin{cases}
				\LC \loglap + q_j \RC u_{f_j} =0 &\text{ in }\Omega, \\
				u_{f_j}=f_j &\text{ in }\Omega_e,
			\end{cases}
		\end{equation}
		for $j=1,2$.
	\end{lemma}

	\begin{proof}
		Via \eqref{eq: self-adjoint}, one can obtain 
		\begin{equation}
			\begin{split}
				\left\langle \LC \Lambda_{q_1}-\Lambda_{q_2} \RC f_1, f_2 \right\rangle  &= \left\langle \Lambda_{q_1}(f_1), f_2 \right\rangle  - \left\langle f_1,  \Lambda_{q_2}(f_2) \right\rangle \\
				&= B_{q_1}(u_{f_1},u_{f_2}) -B_{q_2}(u_{f_1},u_{f_2})\\
				&=\int_{\Omega} \LC q_1 -q_2 \RC u_{f_1}u_{f_2}\, dx,
			\end{split}
		\end{equation}
		where we used \eqref{eq: symmetry log} in the last equality.
	\end{proof}

	\section{Proof of Theorem \ref{thm: uniqueness}}\label{sec: pf of global unique}
	
	As we discussed before, in solving nonlocal inverse problems, one usually needs the Runge approximation for nonlocal operators. The proof of the (qualitative) Runge approximation for $\loglap$ is based on the following unique continuation property (UCP) for the logarithmic Laplacian. Recall the definition of the space $L^1_0(\R^n)$ given in \eqref{L^1_0(R^n)}.
	
	\begin{proposition}[\text{\cite[Theorem 5.1]{CHW2023extension}}]\label{prop: UCP}
		Let $D\subset \R^n$ be a nonempty open set. If $u\in L^1_0(\R^n)$ satisfies 
		\begin{equation}
			u=\loglap u =0 \text{ in $D$ in distributional sense,} 
		\end{equation}
		then $u \equiv 0 $ in $\R^n$.
	\end{proposition}

	Let $\Omega\subset \R^n$ be a bounded nonempty Lipschitz domain, and let $q \in L^\infty(\Omega)$ satisfy \eqref{eigenvalue}. Then the solution operator associated with \eqref{eq: main equation} is defined as 
	\begin{equation}\label{eq: solution op.}
		P_q \vcentcolon \H_T(\Omega_e) \to \mathbb{H}(\R^n), \quad f\mapsto u_f,
	\end{equation}
	where $u_f \in \mathbb{H}(\R^n)$ is the solution to \eqref{eq: main equation}, so the solution of (\ref{eq: well-posedness}) with $F=0$. With Proposition \ref{prop: UCP} at hand, we can show the Runge approximation.
	
	\begin{proposition}[Runge approximation]\label{prop: Runge}
		Let $\Omega \subset \R^n$ be a bounded Lipschitz domain, let $ q\in L^\infty(\Omega)$ satisfy \eqref{eigenvalue}, and consider the solution operator $P_q$ given by \eqref{eq: solution op.}. Then for every nonempty open set $W\Subset \Omega_e$, the set 
		\begin{equation}
			\mathcal{R}\vcentcolon = \left\{ P_q f\big|_{\Omega} : \, f\in C^\infty_c(W)\right\},
		\end{equation}
		is dense in $L^2(\Omega)$.
	\end{proposition}

	\begin{proof}
		By the Hahn-Banach theorem, it suffices to show that for any $w\in L^2(\Omega)$, which satisfies  
		\begin{equation}\label{eq: runge 1}
			\LC P_qf , w\RC_{L^2(\Omega)}=0 \quad \text{for any} \quad f\in C^\infty_c(W),
		\end{equation}
		it results that $w \equiv 0$. Let $\varphi \in \mathbb{H}_0(\Omega)$ be the solution to 
		\begin{equation}\label{eq: runge 2}
			\begin{cases}
				\LC	\loglap +q  \RC \varphi = w &\text{ in }\Omega, \\
				\varphi =0 &\text{ in }\Omega_e.
			\end{cases}
		\end{equation}
		The well-posedness of the above equation was guaranteed in Section \ref{sec: forward problem}. Hence, we have 
		\begin{equation}
			\begin{split}
				B_q (\varphi,f)=B_q(\varphi, f-P_qf) = \LC w , f-P_qf \RC_{L^2(\Omega)} =-\LC w, P_q f \RC_{L^2(\Omega)}=0,
			\end{split}
		\end{equation}
		for any $f\in C^\infty_c(W)$, where we used \eqref{eq: runge 1} in the last identity. Hence for any $f\in C^\infty_c(W)$ we have, since $f \equiv 0$ in $\Omega$,
		\begin{align*}
			\int_{\R^n} \varphi \loglap f\,dx &= 2 \int_{\R^n}(\log |\xi|) \widehat \varphi(\xi)\widehat f(\xi)\,d\xi\\
			&= B_q(\varphi,f)+ \int_{\Omega}q(x) \varphi(x) f(x) \, dx \\
			&= 0.	
		\end{align*}
		This yields that $\loglap\varphi\equiv  0$ in $W$ in distributional sense, while also $\varphi \equiv 0$ in $W$. Note that $\varphi\in L^1_0(\R^n)$ since $\mathbb{H}_0(\Omega) \subset L^1_0(\R^n)$, and this allows us to apply the UCP of Proposition \ref{prop: UCP}. Therefore, Proposition \ref{prop: UCP} implies that $\varphi\equiv 0$ in $\R^n$ and therefore  $w\equiv 0$ as desired. This concludes the proof.	
	\end{proof}

	\begin{remark}
		As we have mentioned earlier in Section \ref{sec: introduction}, the logarithmic Laplacian is a near zero-order operator, so we do not expect that the above-mentioned Runge approximation possesses a higher regularity approximation property.
	\end{remark}
	
	We are now ready to prove the global uniqueness result by using the Runge approximation.

	\begin{proof}[Proof of Theorem \ref{thm: uniqueness}]
		With the condition \eqref{eq: same DN map in thm 1}, the symmetry of the operators $\Lambda_{q_i}$ and integral identity \eqref{eq: integral id} imply 
		\begin{align}\label{eq: integral id in proof 1}
			\int_{\Omega} \LC q_1 -q_2 \RC u_{f}u_{g}\, dx=0 \qquad \text{for all $f \in C^\infty_c(W_1)$, $g \in C^\infty_c(W_2)$,}
		\end{align}
		where $u_{f} = P_{q_1}f$ and $u_{g} = P_{q_2} g \in \mathbb{H}(\R^n)$ are the solutions to
		\begin{equation}\label{eq: equation in proof j=1}
			\begin{cases}
				\LC \loglap + q_1 \RC u_{f} =0 &\text{ in }\Omega, \\
				u_{f}=f  &\text{ in }\Omega_e,
			\end{cases}
		\end{equation}
		and 
		\begin{equation}\label{eq: equation in proof j=2}
			\begin{cases}
				\LC \loglap + q_2 \RC u_{g} =0 &\text{ in }\Omega, \\
				u_{g}=g &\text{ in }\Omega_e,
			\end{cases}
		\end{equation}
		respectively. 
		By the Runge approximation in Proposition \ref{prop: Runge}, given any $h \in L^2(\Omega)$, there exist sequences of functions $f_k \in C^\infty_c(W_1)$, $g_k \in C^\infty_c(W_2)$ with $P_{q_1} f_k \to h$ and $P_{q_2}g_k \to 1$ in $L^2(\Omega)$ as $k\to \infty$. By (\ref{eq: integral id in proof 1}), we then conclude that 
		\begin{equation}
			\int_{\Omega} \LC q_1 -q_2 \RC h \, dx=\lim_{k \to \infty}\int_{\Omega} \LC q_1 -q_2 \RC (P_{q_1}f_k)(P_{q_2}g_k) \, dx =  0.
		\end{equation}
		Due to the arbitrariness of $h\in L^2(\Omega)$, there must hold $q_1=q_2$ a.e. in $\Omega$. This proves the assertion.
	\end{proof}

	\begin{remark}
		It would be interesting to ask if one can prove Theorem \ref{thm: uniqueness} using a single measurement. It is clear that if the potential $q$ is regular enough (e.g. continuous potentials), one may directly apply the existing UCP of Proposition \ref{prop: UCP} for $\loglap$. However, for the rough potential case, one needs to study a measurable UCP result for the logarithmic Laplacian, i.e., the UCP for $\loglap$ holds on a set of positive measures. 
	\end{remark}

	\section{Constructive uniqueness}\label{sec: mono}
	
	To prove Theorem \ref{thm: const. uniqueness}, we will utilize the monotonicity method combined with the localized potentials for \eqref{eq: main equation}. This shows that increasing the potential $q$ increases the DN map $\Lambda_q$ in the sense of quadratic forms, and vice versa. To accomplish the complete arguments, we need the following localized potentials.
	
	\subsection{Localized potentials}
	
	With the Runge approximation at hand, we can immediately derive the existence of the localized potentials.
	
	\begin{lemma}[Localized potentials]\label{lem:localized_potentials} 
		Let $\Omega\subset \R^{n}$ be a bounded Lipschitz domain $q\in L^{\infty}(\Omega)$, and $W\Subset\Omega_{e}$ be a nonempty open set. For every measurable set $M\subset\Omega$, there exists a sequence $f^{k}\in C_{c}^{\infty}(W)$, such that the corresponding
		solutions $u^{k}\in \mathbb H(\R^{n})$ of 
		\begin{equation}\label{eq:lem_loc_pot_uk}
			\begin{cases}
				\LC \loglap +q \RC u^{k}=0 & \text{ in }\Omega,\\
				u^{k}=f^k & \text{ in }\Omega_e,
			\end{cases}
		\end{equation}
		for all $k\in \N$, satisfy that 
		\[
		\int_{M} \left|u^{k} \right|^{2}\, dx \to\infty\quad\text{and}\quad\int_{\Omega\setminus M}\left|u^{k}\right|^{2}\, dx \to0 \quad  \mbox{as}\quad k\to \infty.
		\]
	\end{lemma}
	\begin{proof}
		Applying the Runge approximation of Proposition \ref{prop: Runge}, we find a sequence of functions $\widetilde{f}^{k} \in C_{c}^{\infty}(W)$
		so that the corresponding solutions $\left. \widetilde{u}^{k}\right|_{\Omega}$
		converge to $\frac{\chi_{M}}{\sqrt{\int_{M}1\, dx}}$ in $L^{2}(\Omega)$, and 
		\[
		\left\|\widetilde{u}^{k} \right\|_{L^{2}(M)}^2=\int_{M}\left|\widetilde{u}^{k}\right|^{2}\, dx\to 1,
		\quad\text{and}\quad \left\|\widetilde{u}^{k}\right\|_{L^{2}(\Omega\setminus M)}^2=\int_{\Omega\setminus M}\left|\widetilde{u}^{k} \right|^{2}\,dx \to 0
		\]
		as $k\to \infty$.
		We may assume for all $k\in \mathbb{N}$ that
		$\widetilde{u}^{k}\not\equiv0$ without loss of generality, so that  $\left\|\widetilde{u}^{(k)}\right\|_{L^{2}(\Omega \setminus M)}>0$
		follows from the UCP of Proposition \ref{prop: UCP} for $\loglap$. Taking normalizing 
		\[
		f^{k}:=\frac{\widetilde{f}^{k}}{\left\|\widetilde{u}^{k}\right\|_{L^{2}(\Omega\setminus M)}^{1/2}} , 
		\]
		the sequence of corresponding solutions $u^{k}\in \mathbb H(\R^{n})$ of \eqref{eq:lem_loc_pot_uk} possesses the desired property 
		\[
		\big\|u^{k}\big\|_{L^{2}(M)}^2 = \frac{\left\|\widetilde{u}^{k}\right\|_{L^{2}(M)}^2}{ \left\|\widetilde{u}^{k}\right\|_{L^{2}(\Omega\setminus M)} } \to \infty,
		\quad\text{and}\quad \big\|u^{k}\big\|_{L^{2}(\Omega\setminus M)}^2=\big\|\widetilde{u}^{k}\big\|_{L^{2}(\Omega\setminus M)}\to 0,
		\]
		as $k\to \infty$. This completes the proof.
	\end{proof}

	\subsection{Monotonicity relations}
	
	Here we complete the proof of Theorems~\ref{thm: equivalent monotonicity} and \ref{thm: const. uniqueness}.
	
	
	\begin{lemma}[Monotonicity relations]
		\label{Lemma for monotonicity} Let $\Omega\subset \R^{n}$ be a bounded open Lipschitz domain and $f \in \H_T(\Omega_e)$. Moreover, for $j=1,2$, let $q_{j}\in L^\infty(\Omega)$ satisfy \eqref{eigenvalue}, and let $u_{j}\in \mathbb H(\R^{n})$ be the unique solutions of 
		\begin{equation}\label{eq: log-equation for monotonicity}
			\begin{cases}
				\LC \loglap+q_{j} \RC u_{j}=0 & \mbox{ in }\Omega,\\
				u_{j}=f &\text{ in }\Omega_e.
			\end{cases}
		\end{equation}
		Then we have the monotonicity	relations 
		\begin{equation}\label{monotone relation 1}
			\left\langle \left(\Lambda_{q_{2}}-\Lambda_{q_{1}} \right) f,f\right\rangle \leq\int_{\Omega}(q_{2}-q_{1})\left|u_{1}\right|^{2}\, dx,
		\end{equation}
		and 
		\begin{equation}\label{monotone relation 2} 
			\begin{split}
				\left\langle (\Lambda_{q_2}-\Lambda_{q_1} f, f \right\rangle \geq\int_{\Omega}\left(q_{2}-q_{1}\right)\left|u_{2}\right|^{2} \, dx.
			\end{split}
		\end{equation}
	\end{lemma}
	\begin{proof}
		The definition of the DN map \eqref{eq: defi DN map} implies that
		$$
		\left\langle\Lambda_{q_1}f,f \right\rangle  =B_{q_{1}}(u_{1},u_{1})\qquad \text{and}\qquad \left\langle\Lambda_{q_2}f,f \right\rangle  =B_{q_{2}}(u_{2},u_{2})=B_{q_{2}}(u_{2},u_{1}).
		$$
		By the symmetry property of the bilinear form, we thus have
		\begin{equation}
			\begin{split}
				0&\leq  B_{q_{2}}(u_{2}-u_{1},u_{2}-u_{1})\\
				&=B_{q_{2}}(u_{2},u_{2})-2B_{q_{2}}(u_{2},u_{1})+B_{q_{2}}(u_{1},u_{1})\\
				&= -\left\langle\Lambda_{q_2}f,f \right\rangle+B_{q_{2}}(u_{1},u_{1})\\
				&=\left\langle \LC \Lambda_{q_1}-\Lambda_{q_2}\RC f,f \right\rangle+B_{q_{2}}(u_{1},u_{1})-B_{q_{1}}(u_{1},u_{1})\\
				&= \left\langle\LC \Lambda_{q_1}-\Lambda_{q_2}\RC f,f \right\rangle+\int_{\Omega}(q_{2}-q_{1}) \left| u_{1}\right|^{2} \, dx,
			\end{split}
		\end{equation}
		which proves \eqref{monotone relation 1}. Interchanging $q_{1}$ and $q_{2}$ in the above computations yields \eqref{monotone relation 2}. This proves the assertion.
	\end{proof}

	\begin{proof}[Proof of Theorem~\ref{thm: equivalent monotonicity}]
		We have to show the equivalence of Properties (i)--(iii). Suppose first that (i) holds, i.e., we have $q_{1}\leq q_{2}$ a.e. in $\Omega$. Then (\ref{monotone relation 2}) readily implies that $\left\langle\LC \Lambda_{q_2}-\Lambda_{q_1}\RC f,f \right\rangle \ge 0$ for every $f \in \H_T(\Omega_e)$, so (ii) holds. From (ii) it directly follows that
		(\ref{quadratic sense}) holds for any nonempty bounded open set $W \Subset \Omega_e$, since $C^\infty_c(W) \subset \H_T(\Omega_e)$. Therefore (ii) implies (iii).
		Finally, let us show the implication (iii) $\Longrightarrow$ (i). So we assume that (\ref{quadratic sense}) holds for any nonempty bounded open set $W \Subset \Omega_e$ for some nonempty bounded open set $W \Subset \Omega_e$, and we need to show that $q_{1}\leq q_{2}$ a.e.\ in $\Omega$. Arguing by contradiction, let us assume that $q_{1}\leq q_{2}$ does not hold a.e.\ in $\Omega$. Then there
		exists $\delta>0$ and a positive measurable set $M\subset\Omega$ such that $q_{1}-q_{2}\geq\delta>0$ on $M\subseteq\Omega$. Using the sequence
		of localized potentials from Lemma \ref{lem:localized_potentials}
		for the coefficient $q_{1}$, and the monotonicity inequality from
		Lemma \ref{Lemma for monotonicity}, one can get 
		\begin{equation}
			\begin{split}
				\left\langle \LC \Lambda_{q_2}-\Lambda_{q_1}\RC f^{k},f^{k}\right\rangle  & \leq\int_{\Omega}(q_{2}-q_{1})\left|u_{1}^{k}\right|^{2}\, dx\\
				& \leq \left\|q_{2}-q_{1}\right\|_{L^{\infty}(\Omega\setminus M)}\big\|u_{1}^{k}\big\|_{L^{2}(\Omega\setminus M)}^{2}-\delta \big\|u_{1}^{k}\big\|_{L^{2}(M)}^2\\
				&\to-\infty,\text{ as }k\to \infty,
			\end{split}
		\end{equation}
		which contradicts (\ref{quadratic sense}). Therefore, the assertion is proved.
	\end{proof}
	
	Finally, let us prove Theorem \ref{thm: const. uniqueness}.
	
	\begin{proof}[Proof of Theorem \ref{thm: const. uniqueness}]
		By \cite[Lemma 4.4]{harrach2017nonlocal-monotonicity}, it is known that given any nonnegative $q\in L^\infty(\Omega)$, then one has 
		\begin{equation}\label{eq: simple function sup}
			q(x)=\sup\left\{ \varphi(x):\, \varphi \in \Sigma_{+,0} , \  \varphi \leq q  \right\} \text{ for a.e. }x\in \Omega. 
		\end{equation}
		Moreover, since we assume that $\lambda_1(\Omega)>0$, the relation
		(\ref{eigenvalue}) is satisfied, and it is also satisfied for $\varphi$ in place of $q$ if $\varphi \in \Sigma_{+,0}$.  We may therefore combine \eqref{eq: simple function sup} with Theorem~\ref{thm: equivalent monotonicity}, applied to $q_1= \varphi \in \Sigma_{+,0}$ and $q_2 = q$, to complete the proof.
	\end{proof}

	\begin{remark}
		With the if-and-only-if monotonicity relations, inverse problems have more possible applications. For instance, one can also study inverse obstacle problems, which recover the unknown inclusion via the monotonicity test (see e.g. \cite{harrach2017nonlocal-monotonicity,HPSmonotonicity}). The stability result could also be interesting (see e.g. \cite{harrach2020monotonicity}).
	\end{remark}

	\medskip

	\noindent\textbf{Acknowledgments.} 
	Y.-H. Lin is partially supported by the National Science and Technology Council (NSTC) Taiwan, under the projects 113-2628-M-A49-003. Y.-H. Lin is also a Humboldt research fellow.

	\bibliography{refs} 
	
	\bibliographystyle{alpha}

\end{document}